\documentclass[12pt]{amsart}

\usepackage[dvipdfmx]{graphics}
\usepackage{here}
\usepackage{tikz}
\usetikzlibrary{arrows.meta}

\usepackage{amsmath, amssymb}
\usepackage{pifont}
\usepackage{booktabs}

\numberwithin{equation}{section}
\newtheorem{theorem}{Theorem}[section]  
\newtheorem{theorem?}{``Theorem''}[section]

\newtheorem{lemma}[theorem]{Lemma}

\theoremstyle{definition}

\newtheorem{question}[theorem]{Question}

\theoremstyle{remark}
\newtheorem{remark}[theorem]{Remark}

\newcommand{\R}{{\mathbb R}}
\newcommand{\C}{{\mathbb C}}

\newcommand{\N}{{\mathbb N}}
\newcommand{\Z}{{\mathbb Z}}

\begin{document}
\title[Local zeta functions]
{
Analytic continuation of 
local zeta functions to \\ 
a slit complex plane: 
the one-dimensional case
%On local zeta functions 
%attached to flat functions
} 
\author{Joe Kamimoto, Hiromichi Mizuno
and Toshihiro Nose}

\address{Faculty of Mathematics, Kyushu University, 
Motooka 744, Nishi-ku, Fukuoka, 819-0395, Japan} 
\email{joe@math.kyushu-u.ac.jp}

\address{Faculty of Science and Engineering, 
Kyushu Sangyo University
3-1 Matsukadai 2-chome, Higashi-ku,
Fukuoka, 813-8503, Japan}
\email{h.mizuno@ip.kyusan-u.ac.jp}

\address{Faculty of Engineering, 
Fukuoka Institute of Technology,
Wajiro-higashi 3-30-1, Higashik-ku,
Fukuoka, 811-0295, Japan}
\email{nose@fit.ac.jp}

%%%

%\date{}

%\vspace{9 em}

\begin{abstract}
In this paper, we provide a sufficient condition 
under which the one-dimensional local zeta function 
admits a holomorphic continuation to the complex plane 
cut along the negative real axis.
\end{abstract}

\maketitle

%\clearpage

%\setcounter{tocdepth}{1}
%\tableofcontents

%%%%%%%%%%%%%%%%%%%%%%%
%\setcounter{section}{-1}
%%%%%%%%%%%%%%%%%%%%%%%%%%%%
%%%%%%%%%%%%%%%%%%%%%%%%%%%
\section{Introduction}
%%%%%%%%%%%%%%%%%%%%%%%%%%
%%%%%%%%%%%%%%%%%%%%%%%%%%%
Let us
consider the integrals of the 
following form:
\begin{equation}\label{eqn:1.1}
Z_{\phi}(s;f)=\int_{\R^n} 
|f(x)|^s \phi(x) dx
\end{equation}
for $s\in\C$, 
where $f$ and $\phi$ are 
real-valued continuous functions defined 
on an open neighborhood $U$ of the origin in 
$\R^n$ and 
the support of $\phi$ is contained 
in $U$. 
Since the above integral 
converges locally uniformly on the half-plane
${\rm Re}(s)>-c_0(f)$, where
\begin{equation}
c_0(f)=\sup\{
c\geq 0:|f|^{-c}
\mbox{ in $L^1$ on a neighborhood 
of the origin in } \R^n\},
\end{equation}
$Z_{\phi}(s;f)$ can be regarded 
as a holomorphic function there, 
which is called 
{\it local zeta function} attached to $(f,\phi)$. 
Note that the above $c_0(f)$ 
is a famous quantity called as
{\it log canonical threshold}
or 
{\it extendibility exponent} of $f$. 
More precisely, 
it is known (cf. \cite{KaN19}) that
if $\phi(0)>0$ and $\phi(x)\geq 0$ on $U$, 
then $Z_{\phi}(s;f)$ 
can not be holomorphically 
extended in any small neighborhood 
of $s=-c_0(f)$; in other words, 
$Z_{\phi}(s;f)$
always has a singularity at 
$s=-c_0(f)$. 
It is an interesting issue to investigate
the situation of analytic continuation of 
local zeta functions 
to a wider region and 
there have been many important 
investigation of this issue.

In this paper, we assume that $f(0)=0$.
Otherwise, $Z_{\phi}(s;f)$ is an entire function 
and $c_0(f)=\infty$.
We note that previous studies \cite{CKN13, KaN16tams} 
have shown that, 
under the condition $\phi(0)=0$, 
the singularities of $Z_{\phi}(s;f)$ may shift 
to the left in the complex plane.

In the case where $f$ is real analytic and 
$\phi$ is $C^{\infty}$-smooth, 
it is shown in \cite{BeG69, Ati70} that 
local zeta functions admit a meromorphic continuation 
to the whole complex plane. 
In particular, if $\phi(0)>0$ and $\phi(x)\ge 0$ on $U$, 
then $Z_{\phi}(s;f)$ has a pole at $s=-c_0(f)$ and, 
more precisely, all the poles lie on the negative real axis. 
This result can be proved by applying 
Hironaka's theorem \cite{Hir64}
on resolution of singularities.

On the other hand, 
the purely smooth case cannot 
be treated in the same way, 
since Hironaka's theorem belongs to the real analytic setting 
and is therefore inapplicable in general. 
In \cite{KaN16jmsut}, 
a certain class of $C^{\infty}$ functions, 
including Denjoy-Carleman quasianalytic classes,  
%including the non-flat function of one variable, 
is introduced, 
and for this class the corresponding local zeta functions 
admit a meromorphic continuation analogous 
to that in the real analytic case.

However, 
it is shown in \cite{KaN19, KaM24}  
that there exist specific $C^{\infty}$ 
functions with a flat term 
for which the associated local zeta functions 
possess a non-polar singularity
at $s=-c_0(f)$ (see also \cite{Nos25}).  
When $f$ contains non-zero flat terms, 
it has been shown in 
\cite{Gre06, KaN19, KaN20, Kam24} 
that local zeta functions 
can be holomorphically continued 
to the half-plane ${\rm Re}(s)>-c_0(f)$ in 
many cases.
These results indicate that 
analytic continuation beyond the line 
${\rm Re}(s)=-c_0(f)$ is delicate, 
and the following question naturally arises.

%%%
\begin{question}
Under what conditions can the integral $Z_{\phi}(s;f)$
be holomorphically continued to the complex plane cut along 
$(-\infty,-c_0(f)]$?
\end{question}
%%%

To the best of our knowledge, 
there is no example of continuous functions 
$f$ and $\phi$ in which 
$Z_{\phi}(s;f)$ does not satisfy the above property.
Our main result offers an answer 
to the above question in the one-dimensional case.

%%%%%%%%%%%%%%%%%%%%%%%%%%%%%%%%%%%%%
\begin{theorem}\label{thm:1.2}
Let $n=1$ in \eqref{eqn:1.1}
and let $\delta>0$.
Suppose that $f$ satisfies 
the conditions
\begin{enumerate}
\item 
$f$ is $C^1$ on 
$(-\delta,\delta)\setminus\{0\}$ and 
$f'(x)\neq 0$ for all $x\neq 0$; 
\item
the limits
$
\lim_{x\to +0}f(x)/f'(x)
$
and
$
\lim_{x\to -0}f(x)/f'(x)
$
exist.
\end{enumerate}
%and that $\phi$ is continuous on $(-\delta,\delta)$. 
If the support of $\phi$ is sufficiently small,
then $Z_{\phi}(s;f)$ 
admits a holomorphic continuation
to the complex plane cut along
%can be holomorphically continued 
%to the complex plane  cut along 
$(-\infty,-c_0(f)]$.
\end{theorem}
%%%%%%%%%%%%%%%%%%%%%%%%%%%

Compared with earlier works, 
our result has two main novelties.
First, 
while previous results require 
at least $C^{\infty}$-smoothness 
of both $f$ and $\phi$, 
our theorem assumes only very weak regularity.
%%%
Second, 
even in the case where $f$ and $\phi$ 
are $C^{\infty}$-smooth, 
our theorem applies to situations 
in which $f$ is flat at the origin, 
in the sense that $f^{(k)}(0)=0$ 
for all $k\in\N$.
Note that when $f$ is not flat at the origin, 
not only does $f$ satisfy conditions (i) and (ii), 
but also 
the holomorphic continuation asserted 
in the theorem can be obtained directly.
This turns out to be a delicate issue, and 
it will be discussed in detail in the final section.

This paper is organized as follows.
In Section 2, 
we study the analytic continuation 
of the Mellin transform of continuous functions.
In Section 3, 
we prove the main results.
In Section 4, 
we discuss the assumptions of the main theorem.

\medskip

{\it Notations.}
\begin{itemize}
\item 
%We denote by $\N$, $\R$, $\C$ the set consisting of 
%all natural numbers, real numbers, 
%complex numbers, respectively. Moreover, 
We denote by $\Z_+$, $\R_+$, $\R_{>0}$
the set consisting of all nonnegative integers, 
all nonnegative real numbers, all positive real numbers,  
respectively. 
\item For a $C^{\infty}$ function $f$ and $k\in\N$, 
$f^{(k)}$ is the $k$-th derivative of $f$.
\end{itemize}

%%%%%%%%%%%%%%%%%%%%%%%%%%%

\section{The Mellin transform of continuous functions}

Let $\psi$ be a continuous function 
defined on $(0,\infty)$ whose support is contained 
in $(0,R)$ for some $R \in (0,1)$.
The Mellin transform of $\psi$ is defined by
\begin{equation}\label{eqn:2.1}
{\mathcal M}[\psi](s)
%=F_{\psi}(s)
=\int_0^{\infty} x^{s-1}\psi(x)\,dx
\end{equation}
for $s \in \C$. 
Let $a\in\R$. 
If $\psi(x)=O(x^{a})$ as 
$x\to +0$, then 
the integral in (\ref{eqn:2.1}) converges locally uniformly
on the half-plane ${\rm Re}(s)>-a$, 
which implies that 
${\mathcal M}[\psi](s)$ can be 
regarded as a holomorphic function 
on this domain. 
In the case of the Mellin transform, 
it is also an interesting issue to determine
how wide it can be holomorphically continued
and the following question naturally arises. 
%%%
\begin{question}
Under what conditions can  
${\mathcal M}[\psi](s)$
be holomorphically continued to the complex plane cut along 
$(-\infty,-a]$?
\end{question}
%%%
In the case where $\psi$ is expressed as 
$\psi(x)=x^a u(x)$ where $u(x)$ is a 
$C^{\infty}$ function on 
$(-\epsilon, \infty)$ with some $\epsilon>0$, 
${\mathcal M}[\psi](s)$ always admits a
holomorphic continuation as in Question~2.1.
Indeed, 
since integration by parts yields the following 
equation in the region ${\rm Re}(s)>0$:
\begin{equation}\label{eqn:2.2}
{\mathcal M}[\psi](s)
%=F_{\psi}(s)
=\frac{(-1)^k}{(s+a)(s+a+1)\cdots(s+a+k)}
\int_0^{\infty} x^{s+a+k} u^{(k+1)}(x)\,dx, 
\end{equation}
%where $\psi^{(k)}$ is the $k$-derivative of $\psi$. 
the convergence of the integral 
in (\ref{eqn:2.2}) implies 
that 
the integral can be holomorphically 
continued to the region ${\rm Re}(s)>-a-k-1$
for any $k\in\N$ and, 
moreover, letting $k\to \infty$ in (\ref{eqn:2.2}), 
we see that 
${\mathcal M}[\psi](s)$ can be meromorphically 
continued to the whole complex plane, and 
its poles are contained in the set 
$\{-a-k:k\in\N\}$.  

The following theorem provides an answer 
to Question~2.1 in the more general case. 

%%%%%%%%%%%%%%%%%%%%%%%%%%%%%%%%%%%%%%%
\begin{theorem}
If the limit $\lim_{x\to+0}\psi(x)/x^a$ exists,  
then the Mellin transform 
${\mathcal M}[\psi](s)$ can be holomorphically 
continued to 
the complex plane cut along 
$(-\infty,-a]$. 
\end{theorem}
%%%%%%%%%%%%%%%%%%%%%%%%%%%%%%%%%%%%%%%
\begin{remark}
Let us consider the following three conditions:
\begin{enumerate}
\item[(A)] 
The limit $\lim_{x\to+0}\psi(x)/x^a$ exists; 
\item[(B)] 
There exists a continuous function $g$ on $[0,R)$ such that 
$\psi(x)=x^a g(x)$; 
\item[(C)]
$\psi(x)=O(x^a)$ as $x\to+0$. 
\end{enumerate}
Then it is easy to see the implications: 
(A) $\Longleftrightarrow$
(B) $\Longrightarrow$
(C).
We do not know whether the assumption 
of Theorem~2.1 can be replaced by (C).
\end{remark}

%%%%%%%%%%%%%%%%%%%%%%%%%%%%%%%%%%%%%%%
For a simplicity, we hereafter denote 
$F(s):={\mathcal M}[\psi](s)$. 
Let $s_0$ be a point in
the region ${\rm Re}(s)>-a$. 
Let us consider the Taylor series of 
$F(s)$ at $s_0$
%%%
\begin{equation}\label{eqn:2.3}
\sum_{k=0}^{\infty}
\frac{F^{(k)}(s_0)}{k!}
(s-s_0)^k.
\end{equation}
%where $F^{(k)}$ is the $k$-th derivative of $F$. 
%%%
Since $F$ is holomorphic on 
${\rm Re}(s)>-a$, 
the radius of convergence of 
the power series (\ref{eqn:2.3}) is greater than 
or equal to $a+{\rm Re}(s_0)$.
Theorem~2.2 can be easily obtained 
by using the following lemma. 

%%%%%%%%%%%%%%%%%%%%%%%%%%%%%%%%%%%%%%
\begin{lemma}\label{lem:2.2}
The radius of convergence of 
the power series (\ref{eqn:2.3}) is greater than 
or equal to $|s_0+a|$.
\end{lemma}
%%%%%%%%%%%%%%%%%%%%%%%%%%%%%%%%%%%%%%

\begin{proof}
It suffices to deal with 
the case where $a=0$.   
In this case, it may be assumed that 
$\psi$ is continuous on $[0,R)$ 
(see Remark~2.3). 
When $s_0$ is a positive real number, 
the assertion is trivial.
Without loss of generality, 
we assume ${\rm Im}(s_{0})>0$. 
In order to show the assertion, 
it suffices to prove the following estimate 
by the Cauchy-Hadamard theorem
%%%%%%%%%%%%%%%%%%%%%%%%%%%%
\begin{equation}\label{eqn:2.4}
\limsup_{k \to \infty}
\left| \frac{F^{(k)}(s_0)}{k!} \right|^{1/k} 
\le \frac{1}{|s_0|}.
\end{equation}
%%%%%%%%%%%%%%%%%%%%%%%%%%%%
From the Lebesgue convergence theorem,
the $k$-th derivative of $F(s)$ is expressed as
\begin{equation}\label{eqn:2.5}
F^{(k)}(s)
=\int_0^R x^{s-1}(\log x)^k\psi(x)\,dx.
\end{equation}
Changing the integral variable by $x=e^{-y}$,
we have 
\begin{equation}\label{eqn:2.6}
F^{(k)}(s)
=(-1)^k\int_r^{\infty} y^k e^{-sy}\psi(e^{-y})\,dy,
\end{equation}
where $r:=-\log R>0$.

First, we will show the estimate (\ref{eqn:2.4}) 
under the assumption that  
$\psi$ can be extended to a real analytic function 
near the origin.
Then there exist a positive real number $\delta$ 
satisfying $0<\delta<R$
and
a holomorphic function $\Psi$ defined on the disk 
$\{ z \in \C : |z|<\delta \}$
such that 
the restriction of $\Psi$ 
to $(0,\delta)$ is $\psi$.
We prepare the four integral curves as follows.
\begin{equation*}\label{eqn:}
\begin{split}
&C_1 : z=t \,\,\, 
(\rho_1 \le t \le \rho_2), \\
&C_2 : z=\rho_1 e^{i\theta} \,\,\, 
(-\theta_0 \le \theta \le 0), \\
&C_3 : z=te^{-i\theta_0} \,\,\, 
(\rho_1 \le t \le \rho_2), \\
&C_4 : z=\rho_2 e^{i\theta} \,\,\, 
(-\theta_0 \le \theta \le 0),  
\end{split}
\end{equation*}
where $\theta_0:={\rm Arg}(s_0) \in (0,\pi/2)$
and $\rho_1, \rho_2$ are positive real numbers with 
$-\log \delta/\cos \theta_0<\rho_1<\rho_2$.
Each curve admits a direction,
which can be determined as in the Figure~1.
The anti-clockwise oriented closed curve 
constructed from the oriented curves 
$C_1, C_2, C_3, C_4$ is denoted by $C$
and 
the bounded domain surrounded 
by $C$ is denoted by $D$.

%%%%%%%%%%%%%%%%%%%%%%%%%%%%%%%%%%%%%%%%%
%%%%%%%%%%%%%%%%%%%%%%%%%%%%%%%%%%%%%%%%%
%\begin{center}
%\begin{tikzpicture}%[x=0.55cm,y=0.54cm]

\begin{figure}[htbp]
\centering
\begin{tikzpicture}

\draw[thick,-stealth] (-1,0)--(5,0) node [anchor=west]{$\mathrm{Re}$};
\draw[thick,-stealth] (0,-5)--(0,1) node [anchor=south]{$\mathrm{Im}$};

\node at (1,0) [above] {$r$};
\node at (2,0) [above] {$\rho_1$};
\node at (4,0) [above] {$\rho_2$};

\filldraw[fill=black] (1,0) circle[radius=0.5mm];
\filldraw[fill=black] (2,0) circle[radius=0.5mm];
\filldraw[fill=black] (4,0) circle[radius=0.5mm];
\filldraw[fill=black] (1,{-sqrt(3)}) circle[radius=0.5mm];
\filldraw[fill=black] (2,{-2*sqrt(3)}) circle[radius=0.5mm];

\draw[dashed] (0,0) -- ({5/sqrt(3)},-5);

%%%%D%%%%%
\fill[black!13] (2,0) arc(0:-60:2) -- (1,{-sqrt(3)}) -- (2,{-2*sqrt(3)}) 
-- (2,{-2*sqrt(3)}) arc(-60:0:4) -- (2,0);

\node at (2.7,-1.5) {$D$};

%%%%C1%%%%
\draw[very thick,-stealth] (4,0)--(3,0);
\draw[very thick] (3,0)--(2,0);

%%%%C2%%%%
\draw[very thick,-stealth] (2,0) arc(0:-30:2);
\draw[very thick] ({sqrt(3)},-1) arc(-30:-60:2);

%%%%C3%%%%
\draw[very thick,-stealth] (1,{-sqrt(3)})--(3/2,{-3/2*sqrt(3)});
\draw[very thick] (3/2,{-3/2*sqrt(3)})--(2,{-2*sqrt(3)});

%%%%C4%%%%
\draw[very thick,-stealth] (2,{-2*sqrt(3)}) arc(-60:-30:4);
\draw[very thick] ({2*sqrt(3)},-2) arc(-30:0:4);

\node at (3,0) [above] {$C_1$};
\node at ({sqrt(3)},-1) [right] {$C_2$};
\node at (3/2,{-3/2*sqrt(3)}) [left] {$C_3$};
\node at ({2*sqrt(3)},-2) [right] {$C_4$};

\draw[thick] (1/2,0) arc(0:-60:1/2);
\node at ({1/4*sqrt(3)},-1/2) [right] {$-\psi$};

\end{tikzpicture}

\caption{Integral curves}
\label{fig:integral_curves}

\end{figure}
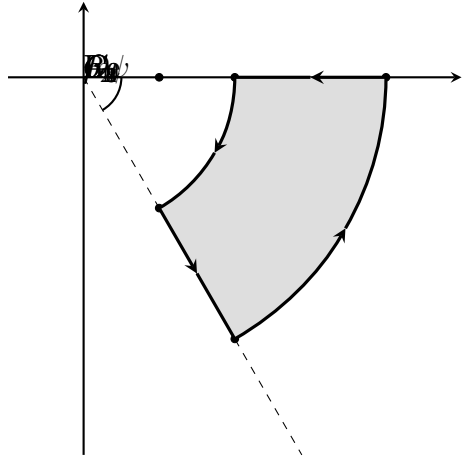

%%%%%%%%%%%%%%%%%%%%%%%%%%%%%%%%%%

%%%%%%%%%%%%%%%%%%%%%%%%%%%%%%%%%%

We define $\Phi_k(z):=z^k e^{-s_0z}\Psi(e^{-z})$.
Since $\Phi_k(z)$ is holomorphic on an open 
neighborhood of $\overline{D}$,
the Cauchy integral theorem implies
\begin{equation}\label{eqn:2.7}
\int_C \Phi_k(z)\,dz
=I_1(k)+I_2(k)+I_3(k)+I_4(k) %\sum_{j=1}^4\int_{C_j}\Phi(z)\,dz
=0,
\end{equation}
where 
$$
I_j(k):=\int_{C_j}\Phi_k(z)\,dz \quad 
\text{ for } j=1,2,3,4.
$$ 
From \eqref{eqn:2.7},
we have the following relationship concerning with 
the integral in \eqref{eqn:2.6}:
\begin{equation}\label{eqn:2.8}
\int_r^{\rho_2} 
y^k e^{-s_0y}\psi(e^{-y})\,dy
=\tilde{I}(k)+I_2(k)+I_3(k)+I_4(k),
\end{equation}
with
\begin{equation*}
\tilde{I}(k)=
\int_r^{\rho_1} y^k e^{-s_0y}\psi(e^{-y})\,dy.
\end{equation*}

{\bf (The behavior of $\tilde{I}(k)$ as $k\to\infty$)} \\
The integral $\tilde{I}(k)$ can be estimated as
\begin{equation*}
\left|\tilde{I}(k)\right|
\le M\int_0^{\rho_1} y^k e^{-{Re}(s_0)y}\,dy 
\le M\int_0^{\rho_1} y^k\,dy
=\frac{M}{k+1}\rho_1^{k+1},
\end{equation*}
where $M:=\max\{|\Psi(e^{-z})|:z\in C\}$.
Then,
we have
\begin{equation*}\label{eqn:2.}
\left|\frac{\tilde{I}(k)}{k!} \right|^{1/k} 
\le \frac{1}{(k!)^{1/k}}\frac{1}{(k+1)^{1/k}} 
M^{1/k}  {\rho_1}^{1+1/k}.
\end{equation*}
The right-hand side in the above 
inequality  tends to $0$ as $k \to \infty$,
which implies
\begin{equation}\label{eqn:2.9}
\lim_{k\to \infty}\left|\frac{\tilde{I}(k)}{k!}\right|^{1/k}
=0.
\end{equation}

{\bf (The behavior of $I_2(k)$ as $k\to\infty$)} \\
The integral $I_2(k)$ can be estimated as
\begin{equation*}
|I_2(k)|
\le M \rho_1^k
\int_{-\theta_0}^0 
e^{-|s_0|\rho_1 \cos(\theta+\theta_0)}\,d\theta
\le M \, \rho_1^k \, 
\theta_0 e^{-|s_0|\rho_1 \cos\theta_0}.
\end{equation*}
Then,
we have
\begin{equation*}
\left|\frac{I_2(k)}{k!} \right|^{1/k} 
\le \frac{1}{(k!)^{1/k}} M^{1/k} \rho_1 
\theta_0^{1/k} e^{-(|s_0|\rho_1 \cos\theta_0)/k}
\le \frac{1}{(k!)^{1/k}} M^{1/k}\rho_1 \theta_0^{1/k}.
\end{equation*}
Since the last term tends to $0$ as $k \to \infty$,
we have
\begin{equation}\label{eqn:2.10}
\lim_{k\to\infty}\left|\frac{I_2(k)}{k!}\right|^{1/k}=0.
\end{equation}

{\bf (The behavior of $I_3(k)$ as $k\to\infty$)} \\
The integral $I_3(k)$ can be estimated as
\begin{equation*}
|I_3(k)|
\le M\int_0^{\rho_2} t^k e^{-|s_0|t}\,dt.
\end{equation*}
By changing the integral variable $u=|s_0|t$,
we have
\begin{equation*}
\begin{split}
|I_3(k)|
\le &\frac{M}{|s_0|^{k+1}}
\int_0^{\rho_2 |s_0|} e^{-u} u^k \,du \\
\le &\frac{M}{|s_0|^{k+1}}
\int_0^{\infty} e^{-u} u^k \,du 
=\frac{M}{|s_0|^{k+1}}\Gamma(k+1).
\end{split}
\end{equation*}
Since $\Gamma(k+1)=k!$,
we have
\begin{equation*}
\left|\frac{I_3(k)}{k!} \right|^{1/k} 
\le \frac{M^{1/k}}{|s_0|^{1+1/k}}. 
\end{equation*}
The right-hand side in the above inequality 
tends to $1/|s_0|$ as $k \to \infty$,
which implies
\begin{equation}\label{eqn:2.11}
\limsup_{k\to \infty}
\left|\frac{I_3(k)}{k!}\right|^{1/k}\leq \frac{1}{|s_0|}.
\end{equation}

{\bf (The bahavior of $I_4(k)$ as $k\to\infty$)} \\
The integral $I_4(k)$ 
can be estimated similarly to the case of 
$I_2(k)$ as
\begin{equation*}
|I_4(k)|
\le M {\rho_2}^k \theta_0 e^{-|s_0|\rho_2 \cos\theta_0}.
\end{equation*}
Then, we have
\begin{equation}\label{eqn:2.12}
\lim_{k\to \infty}\left|\frac{I_4(k)}{k!}\right|^{1/k}
=0.
\end{equation}

From \eqref{eqn:2.6} and \eqref{eqn:2.8}, we see
\begin{equation}\label{eqn:2.13}
\left| \frac{F^{(k)}(s_0)}{k!} \right|^{1/k}
\le \left| \frac{\tilde{I}(k)}{k!} \right|^{1/k}
+\left| \frac{I_2(k)}{k!} \right|^{1/k}
+\left| \frac{I_3(k)}{k!} \right|^{1/k}
+\left| \frac{I_4(k)}{k!} \right|^{1/k}.
\end{equation}
From
\eqref{eqn:2.9}, \eqref{eqn:2.10}, 
\eqref{eqn:2.11} and \eqref{eqn:2.12},
we obtain the estimate \eqref{eqn:2.4}
under the assumption of 
the real analyticity of $\psi$ near the origin.

Next, 
let us consider the general case, that is, 
$\psi$ is only continuous on $[0,R)$.
By the Weierstrass approximation theorem,
for any $\epsilon>0$ and $k \in \N$,
there exists a polynomial $P(x)$ 
satisfying that 
$|\psi(x) - P(x)| < ({\rm Re}(s_0))^{k+1}\epsilon^k$ 
on $[0,R]$.
Let $G(s):={\mathcal M}[P](s)$. 
Then, 
we have
\begin{equation}\label{eqn:2.14}
\begin{split}
\left|F^{(k)}(s_0)\right| - \left|G^{(k)}(s_0)\right|  
\le& \int_0^R x^{{\rm Re}(s_0)-1}
\left|\psi(x) - P(x)\right|\cdot|\log x|^k\,dx \\
<& \,({\rm Re}(s_0))^{k+1}\epsilon^k
\int_0^{1} x^{{\rm Re}(s_0)-1}|\log x|^k\,dx\\
=& \,\epsilon^k\int_0^{\infty}e^{-u} u^k \,du
=\epsilon^k k!.
\end{split}
\end{equation}
We remark that
the first equality in 
\eqref{eqn:2.14} is obtained 
by the change of the variables 
$u=-({\rm Re}(s_0))\log x$.
Hence, we have 
\begin{equation*}
\begin{split}
\limsup_{k \to \infty}
\left| \frac{F^{(k)}(s_0)}{k!} \right|^{1/k} 
&\le \limsup_{k \to \infty}
\left(\left| \frac{G^{(k)}(s_0)}{k!}\right| 
+\epsilon^{k}\right)^{1/k}\\
&\le \limsup_{k \to \infty}
\left| \frac{G^{(k)}(s_0)}{k!} \right|^{1/k} 
+ \epsilon
\leq \frac{1}{|s_0|} + \epsilon.
\end{split}
\end{equation*}
As $\epsilon>0$ is taken arbitrarily, 
the estimate \eqref{eqn:2.4} holds.
\end{proof}

%%%%%%%%%%%%%%%%%%%%%%%%%%%%%%%%%%%%%%%
\section{Proof of Theorem~1.2}

We consider the integral
\begin{equation}\label{eqn:3.1}
{\mathcal Z}(s)=\int_{0}^{\infty}|f(x)|^s \phi(x)dx
\end{equation}
for $s\in\C$,  
where $f, \phi$ are as in \eqref{eqn:1.1}.
Since the integral ${\mathcal Z}(s)$ converges 
locally uniformly on the half-plane ${\rm Re}(s)>-c_{0}(f)$,
the integral ${\mathcal Z}(s)$ becomes 
a holomorphic function there.
In order to show Theorem~1.2, 
it suffices to consider the case of 
$\mathcal{Z}(s)$
%The following theorem is essential.
%%%%%%%%%%%%%%%%%%%%%%%%%%%%%
\begin{theorem}\label{thm:3.1}
Let $f$ satisfy the conditions in Theorem~1.2.
If the support of $\phi$ is sufficiently small,
then
$\mathcal{Z}(s)$ can be holomorphically continued
to the complex plane cut along $(-\infty,-c_0(f)]$.
\end{theorem}

\begin{proof}
Without loss of generality, we assume 
$f'(x)>0$ for $x>0$.
Since $f$ is monotonously increasing, 
the inverse function $f^{-1}$ exists.
Here, 
we take the support of $\phi$ to be sufficiently small such that ${\rm Supp}(\phi)\subset [-r,r]$, 
where $r>0$ satisfies $|f(r)|<1$.
By the change of variables $u=f(x)$,
the integral (\ref{eqn:3.1}) is expressed as
\begin{equation}\label{eqn:3.2}
\mathcal{Z}(s)=\int_0^{\infty} u^{s-1}\tilde{\phi}(u)\,du,
\end{equation}
where 
\begin{equation}
\tilde{\phi}(u)=u\cdot(f^{-1}(u))' 
\cdot \phi(f^{-1}(u)).
\end{equation}\label{eqn:3.3}
We remark that 
the support of $\tilde{\phi}$ is 
contained in $[0, f(r)]$. 
Since the limit of $f(x)/f'(x)$ as $x\to +0$ exists,
$u\cdot(f^{-1}(u))'$ is continuous near the origin,
so is $\tilde{\phi}(u)$. 
Therefore, Theorem~2.2 implies that 
$\mathcal{Z}(s)$ can be holomorphically extended to
the complex plane cut along $(-\infty,-c_0(f)]$.
\end{proof}

Let us prove Theorem~1.2.
We decompose the integral $Z_{\phi}(s;f)$ as
\begin{equation}\label{eqn:3.4}
Z_{\phi}(s;f)=
\int_{0}^{\infty}|f(x)|^s \phi(x)dx+
\int_{0}^{\infty}|f(-x)|^s \phi(-x)dx.
\end{equation}
Applying Theorem~\ref{thm:3.1} 
to the two integrals in (\ref{eqn:3.4}),
we obtain the assertion.

%%%%%%%%%%%%%%%%%%%%%%%%%%%%%%%%%%

\section{Remarks on the assumption
of $f$ in Theorem~1.2}

In this section, 
we always assume that
$f$ is $C^{1}$ 
on $(-\delta, \delta)\setminus\{0\}$ 
with $f(0)=0$, and the support of 
$\phi$ is contained in $(-\delta, \delta)$, 
where $\delta$ is a positive constant. 

First, let us consider the case where 
$f$ is $C^{\infty}$-smooth. 
When
$f$ is not flat at the origin, 
there exist $m\in\N$ and 
$u\in C^{\infty}(-\delta,\delta)$
with $u(0)\neq 0$ such that $f(x)=x^m u(x)$ on 
$(-\delta, \delta)$. 
It is easy to check that this $f$ 
satisfies the conditions in Theorem~1.2.
In fact, we have 
\begin{equation*}
\frac{f(x)}{f'(x)}=
\frac{x}{m u(x)+x u'(x)} \longrightarrow 0 
\quad \text{ as } x\to 0.
\end{equation*}
Moreover, when $\phi$ is $C^{\infty}$-smooth, 
the assertion of the theorem can be shown directly.
Indeed, by a change of variables, 
$f$ can be regarded as a monomial 
$\pm x^m$, which implies that 
$Z_{\phi}(s;f)$ 
admits a meromorphic continuation 
to the whole complex plane and 
that its poles are contained 
on the negative real axis, 
by using \eqref{eqn:2.2}.
%(See Lemma 4.1 and Remark 4.2 of \cite{KaN20}.)
On the other hand, 
even in the case where $f$ is flat at the origin, 
there are many examples in which 
$Z_{\phi}(s;f)$ admits the analytic continuation 
to $\C\setminus(-\infty,0]$. 
The function
\begin{equation}\label{eqn:4.1}
f(x)= e^{-1/|x|^{\alpha}} \text{ for }
x\in\R\setminus\{0\}; \; f(0)=0,
\end{equation}
with $\alpha>0$, 
is a typical example. 
In fact, 
we have 
\begin{equation*}
\frac{f(x)}{f'(x)}=
\frac{1}{\alpha}\, |x|^{\alpha+1} \longrightarrow 0 
\quad \text{ as } x\to 0.
\end{equation*}

Next, 
let us consider the more general cases. 
In order to understand the features 
of the conditions (i) and (ii) in Theorem~1.2, 
we will consider the differences 
among the following four conditions.
%%%
\begin{enumerate}
\item[(A)] 
$f'(x)>0$ for $x\in(0,\delta)$;
\item[(B)] 
$f'(x)>0$ for $x\in(0,\delta)$ and the limit 
$\lim_{x\to 0} f(x)/f'(x)$ exists; 
%(the assumption of Theorem~1.1);
\item[(C)] 
$f\in C^2((0,\delta))$, 
$f'(x)>0$ and $f''(x)\geq 0$ for $x\in(0,\delta)$;
\item[(D)]  
$f\in C^2((0,\delta))$, 
$f''(x)>0$ for $x\in(0,\delta)$.
\end{enumerate}
It is not difficult to see the implications:
\begin{equation*}
{\rm (D)} \Longrightarrow 
{\rm (C)} \Longrightarrow 
{\rm (B)} \Longrightarrow 
{\rm (A)}.
\end{equation*}
Therefore, 
the condition (D), i.e., $f$ is a convex function near the origin, 
is sufficient for the analytic continuation in Theorem~1.2.
Note that 
the function in (\ref{eqn:4.1}) satisfies the condition (D).
Although the differences 
among the above four conditions are subtle, 
the converse of each implication is not true, 
which will be seen by the examples below.

\medskip
%%%%%%%%%%%%%%%%%%%%%%%%%%%%%%%%%%%%%%%
{\bf [(A) $\not\Rightarrow$ (B)] }
\quad 
For each $n\in\N$, 
we define a $C^{\infty}$ function 
$\chi_n$ on $\R$ by 
\begin{equation}\label{eqn:4.2}
\chi_n(x)=
\begin{cases}
&n \quad 
\text{ if } \left|x-\frac{1}{n}\right|<\frac{1}{8n^2},\\
&0 \quad
\text{ if } \left|x-\frac{1}{n}\right|>\frac{1}{4n^2},
\end{cases}
\end{equation}
and, otherwise, $\chi_n$ satisfies 
$0\leq \chi_n(x) \leq n$. 
We remark that 
the supports of $\chi_n$ and $\chi_m$ 
do not intersect if $n\neq m$.
Let 
\begin{equation}\label{eqn:4.3}
f(x):=e^{-g(x)} \text{ for } x \in (0,1); \;\;
f(0)=0,  
\end{equation} 
where 
\begin{equation*}
g(x)=\int_x^1 h(t) dt \quad  \text{with }
h(x)= \sum_{n=1}^{\infty} \chi_n(x)+1.
\end{equation*}
It follows from direct computations 
that the following properties hold:
\begin{itemize}
\item $f,g$ and $h$ are $C^{\infty}$-smooth on $(0,1)$;
\item $\lim_{x\to +0}g(x)=\infty$ and 
$\lim_{x\to +0}f(x)=0$, 
which implies that $f$ extends 
to a continuous function on $[0,1)$; 
\item $f'(x)=-g'(x) e^{-g(x)}=h(x) e^{-g(x)}>0$ on $(0,1)$;
%\item $\lim_{x\to+0} f'(x)=\lim_{x\to+0} h(x) e^{-g(x)}=0$, 
%which implies that $f$ can be regarded as 
%a $C^1$ function on $[0,1)$;
\item The limit of 
$f(x)/f'(x)(=-1/g'(x)=1/h(x))$ 
as $x\to+0$ does not exist. 
\end{itemize} 

We comment on the regularity of $f$ at the origin.
Since $\lim_{x\to 0} f'(x)=\infty$, 
$f$ cannot be extended as a $C^1$ function to 
$[0,\delta)$.
If ``$n$" is replaced by ``$n\log n$" 
in the definition of $\chi_n$ (\ref{eqn:4.2}), 
then $\lim_{x\to +0} f'(x)=0$, and hence 
$f$ extends to a $C^1$ function on $[0,\delta)$.

\medskip
%%%%%%%%%%%%%%%%%%%%%%%%%%%%%%%%%%%%%%%
{\bf [(B) $\not\Rightarrow$ (C)] } \quad
Let $\alpha$ be a positive constant. 
We define 
\begin{equation*}\label{eqn:4.4}
f(x)= e^{-1/|x|^\alpha}\left(1+\sin^2\frac{1}{|x|^{\alpha}}\right)
\text{ for } x \in\R\setminus\{0\}; \;\;
f(0)=0.
\end{equation*}
%%%
It follows from direct computations 
that the following properties hold:
\begin{itemize}
\item $f$ is $C^{\infty}$-smooth on $\R$;
\item $\lim_{x\to +0}f^{(k)}(x)=0$ for all $k\in\N$, 
which implies that $f$ is flat at the origin;  
\item There exists $\delta>0$ such that 
$f'(x)>0$ on $(0,\delta)$;
\item For any $n\in\N$, 
there exist $x_n, y_n\in (0,1/n)$ such that 
$f''(x_n)>0>f''(y_n)$; 
\item The limits of 
$f(x)/f'(x)$ as $x\to\pm 0$ exist. 
\end{itemize}

\medskip
%%%%%%%%%%%%%%%%%%%%%%%%%%%%%%%%%%%%
{\bf [(C) $\not\Rightarrow$ (D)] } \quad
Let $\alpha$ be a positive real number. 
We define a function $h(x)$ on $\R$ by 
\begin{equation*}\label{eqn:4.}
h(x)= e^{-1/|x|^\alpha}\cdot 
\sin^2\frac{1}{|x|^{\alpha}}
\text{ for } x \in\R\setminus\{0\}; \;\;
h(0)=0.
\end{equation*}
Moreover, we define two functions 
$g(x)$ and $f(x)$ on $\R$ by 
\begin{equation*}
g(x)=\int_0^{x} h(t)dt, 
\quad 
f(x)=\int_0^x g(t)dt.
\end{equation*}
It follows from direct computations 
that the following properties hold:
\begin{itemize}
\item $f, g$ and $h$ are $C^{\infty}$-smooth on $\R$;
\item $\lim_{x\to +0}f^{(k)}(x)=0$ for all $k\in\N$, 
which implies that $f$ is flat at the origin;  
\item There exists $\delta>0$ such that 
$f'(x)>0$ on $(0,\delta)$;
\item For any $n\in\N$, 
there exists $x_n\in (0,1/n)$ such that 
$f''(x_n)=0$. 
\end{itemize}

\begin{question}
Is the existence of 
$\lim_{x \to \pm 0} f(x)/f'(x)$
necessary for the analytic continuation 
of $Z_{\phi}(s; f)$ as stated in Theorem~1.2?
\end{question}
\noindent
In particular, 
we do not know whether or not 
${\mathcal Z}(s)$ admits this analytic continuation 
in the case of the function (\ref{eqn:4.3}).

\medskip 

{\it Acknowledgements.} 
This work was supported by JSPS KAKENHI Grant Numbers JP25K07037 and JP26K06847.

%%%%%%%%%%%%%%%%%%%%%%%%%%

\end{document}